\documentclass[11pt]{amsart}
\usepackage[english]{babel}
\usepackage[T1]{fontenc}
\usepackage[latin1]{inputenc}
\usepackage{graphicx}
\usepackage{amsmath,amssymb,amsthm,mathrsfs,txfonts,ifthen,xr}
\usepackage{soul}
\usepackage{array, multirow}
\usepackage{stmaryrd}
\usepackage{mathrsfs}
\usepackage{color}
\usepackage{hyperref}
\usepackage{enumerate}
\usepackage[all]{xy}
\usepackage{tikz-cd}
\usetikzlibrary{decorations.pathreplacing}

\usepackage{ytableau}

\theoremstyle{plain}
\newtheorem{theo}{Theorem}[section]
\newtheorem{lemma}[theo]{Lemma}
\newtheorem{proposition}[theo]{Proposition}
\newtheorem{corollary}[theo]{Corollary}

\theoremstyle{definition}
\newtheorem{definition}[theo]{Definition}

\newtheorem{notation}[theo]{Notation}

\theoremstyle{remark}
\newtheorem{remark}[theo]{Remark}

\def\A{{\rm A}}
\def\B{{\rm B}}
\def\C{{\rm C}}
\def\D{{\rm D}}

\def\F{{\rm F}}
\def\G{{\rm G}}

\def\J{{\rm J}}

\def\M{{\rm M}}

\def\P{{\rm P}}

\def\U{{\rm U}}

\def\X{{\rm X}}
\def\Y{{\rm Y}}

\def\tr{{\rm tr}}

\def\GL{{\rm GL}}	

\def\Id{{\rm Id}}

\def\Mat{{\rm Mat}}

\def\log{{\rm log}}

\def\Spec{{\rm Spec}}

\def\diag{{\rm diag}}

\allowdisplaybreaks
\title{Kubo-Ando means on the cone of $\J$-Hermitian matrices}
\author{Sebastian Micalizzi}
\address{Department of Mathematics and Statistics\\ University of North Florida \\ 1 UNF Drive \\ Jacksonville \\ FL 32224 \\ USA}
\email{n01409569@unf.edu}
\author{Hayden Tyler}
\address{Department of Mathematics and Statistics\\ University of North Florida \\ 1 UNF Drive \\ Jacksonville \\ FL 32224 \\ USA}
\email{n01512071@unf.edu}



\date{}

\begin{document}

\begin{abstract}

The theory of Kubo-Ando means provides a fundamental correspondence between operator means and operator monotone functions on the positive half-line. In this paper, we first extend this theory to the cone of positive quaternionic Hermitian matrices, establishing quaternionic analogues of the main results of Kubo and Ando. We then develop an analogous framework for the cone of $\J$-Hermitian positive matrices associated with an indefinite inner product. In particular, we prove that Kubo-Ando means on this cone are again in one-to-one correspondence with operator monotone functions, thereby extending the classical characterization to this indefinite setting.

\end{abstract}

\maketitle

\tableofcontents

\section{Introduction}

The theory of Kubo-Ando means, introduced by Kubo and Ando in the early 1980s \cite{KUBOANDO}, provides a general framework for studying binary operations on positive definite matrices and positive operators. One of the fundamental achievements of their theory is the characterization of operator means by way of operator monotone functions. More precisely, they proved that every Kubo-Ando mean is uniquely determined by a normalized operator monotone function through a functional calculus formula. This correspondence unifies many classical matrix means, including the arithmetic, geometric, harmonic, logarithmic, and power means, and has become one of the central tools in matrix analysis and operator theory. Since its introduction, Kubo-Ando theory has found numerous applications in operator inequalities, matrix geometry, quantum information theory, mathematical physics, and numerical analysis\,.

\noindent More recently, several authors have investigated extensions of matrix geometric structures beyond the cone of positive definite matrices. In particular, in \cite{JOSEALLAN}, the authors introduced the cone of $\J$-Hermitian positive matrices associated with an indefinite inner product. They showed that many classical geometric constructions on the positive cone admit natural analogues in this new setting, including exponential and logarithm maps, powers, geodesics, and geometric means. These results suggest that the cone $\mathscr{P}_{\J}$ possesses a rich geometric structure closely related to that of the classical positive cone\,.

\noindent The purpose of the present paper is to show that this analogy extends to the theory of Kubo-Ando means. More precisely, we introduce the notion of a $\J$-Kubo-Ando mean and prove that the classical correspondence between operator means and operator monotone functions remains valid in this indefinite setting. Our approach is based on the natural bijection between the cone of positive definite matrices and the cone of $\J$-Hermitian positive matrices, allowing us to transport the classical Kubo-Ando theory almost verbatim. As a consequence, every classical operator mean admits a canonical $\J$-analogue satisfying the same structural properties\,.

\noindent Before studying the indefinite setting, we also consider the quaternionic case. Although the functional calculus for quaternionic Hermitian matrices is well understood, a systematic treatment of Kubo-Ando means over the quaternionic positive cone does not seem to have appeared in the literature. We therefore establish quaternionic analogues of the main results of Kubo and Ando, proving that normalized operator monotone functions again classify operator means. This provides a unified treatment of Kubo-Ando means over the three real division algebras
\begin{equation*}
\mathbb{R},\qquad
\mathbb{C},\qquad
\mathbb{H},
\end{equation*}
before passing to the indefinite $\J$-Hermitian setting\,.

\medskip

\noindent The paper is organized as follows. In Section \ref{SectionPreliminaries}, we recall the necessary background on positive definite matrices, quaternionic Hermitian matrices, and the cone of $\J$-Hermitian positive matrices. We then establish in Section \ref{SectionQuaternions} the quaternionic analogue of the Kubo-Ando correspondence. In Section \ref{SectionKuboAndoPJ}, we introduce $\J$-Kubo-Ando means, prove that they are in one-to-one correspondence with classical Kubo-Ando means, and establish their representation in terms of normalized operator monotone functions. Finally, we illustrate the theory by deriving explicit formulas for several classical means, including the arithmetic, harmonic, power, and geometric means\,.

\bigskip

\noindent \textbf{Acknowledgements:} The authors wish to thank Jose Franco and Allan Merino for their authorship of "The Cone of J-Hermitian Means and a Geometric Mean", which robustly develops important tools used in our research. We are indebted to Dr. Merino for his mentorship during this project. He was instrumental in communicating the motivation and utility of these methods. His perspective as a researcher allowed us to direct our efforts towards meaningful research and exposition of it. Without his willingness to meet our questions with both expertise and enthusiasm, this paper would not exist in its current form\,.

\section{Preliminaries}

\label{SectionPreliminaries}

Let $\mathbb{D} \in \left\{\mathbb{R}\,, \mathbb{C}\right\}$. For all $n \in \mathbb{N}$, we denote by $\Mat_{n}(\mathbb{D})$ the set of $n \times n$ matrices with entries in $\mathbb{D}$, and by $\GL_{n}(\mathbb{D})$ the set of invertible matrices in $\Mat_{n}(\mathbb{D})$. For all $p\,, q \in \mathbb{N}$ such that $p+q = n$, we denote by $\J := \J_{p, q}$ the matrix in $\GL_{n}(\mathbb{D})$ given by
\begin{equation*}
\J = \begin{pmatrix} \Id_{p} & 0 \\ 0 & -\Id_{q}\end{pmatrix}\,.
\end{equation*}

\begin{notation}

Throughout the paper, we adopt the notation of \cite{JOSEALLAN}. For all $\X \in \Mat_{n}(\mathbb{D})$, we denote by $\X^{*}$ and $\X^{\sharp}$ the matrices respectively given by
\begin{equation*}
\X^{*} := \overline{\X}^{t}\,, \qquad \qquad \X^{\sharp} := \J\X^{*}\J\,.
\end{equation*}

\end{notation}

\noindent We denote by $\U$ and $\U_{\J}$ the subgroups of $\GL_{n}(\mathbb{D})$ given by
\begin{equation*}
\U = \left\{g \in \GL_{n}(\mathbb{D})\,, gg^{*} = \Id_{n}\right\}\,, \qquad \U_{\J} = \left\{g \in \GL_{n}(\mathbb{D})\,, gg^{\sharp} = \Id_{n}\right\}\,,
\end{equation*}

\begin{definition}

Let $\A \in \Mat_{n}(\mathbb{D})$. We say that $\A$ is positive (resp. positive semi-definite) and write $\A > 0$ (resp. $\A \geq 0$) if for all non-zero $x \in \mathbb{D}^{n}$, we get $x^{*}\A x > 0$ (resp. $x^{*}\A x \geq 0$)\,.

\end{definition}

\noindent We denote by $\mathfrak{p}$ and $\mathscr{P}$ the subsets of $\Mat_{n}(\mathbb{D})$ and $\GL_{n}(\mathbb{D})$ respectively given by
\begin{equation*}
\mathfrak{p} = \left\{\X \in \Mat_{n}(\mathbb{D})\,, \X = \X^{*}\right\}\,, \qquad \mathscr{P} = \left\{\X \in \mathfrak{p}\,, \X > 0\right\}\,.
\end{equation*}
It is well known (see \cite{BHATIA}) that the exponential map
\begin{equation*}
\exp: \Mat_{n}(\mathbb{D}) \ni \X \to \exp(\X) = \sum\limits_{k = 0}^{\infty} \frac{\X^{k}}{k!} \in \GL_{n}(\mathbb{D})
\end{equation*}
is such that $\exp(\mathfrak{p}) \subseteq \mathscr{P}$. Moreover, the map
\begin{equation*}
\exp: \mathfrak{p} \to \mathscr{P}
\end{equation*}
is bijective. We denote by 
\begin{equation*}
\log: \mathscr{P} \to \mathfrak{p}
\end{equation*}
the inverse of $\exp$\,.

\begin{definition}

For all $t \in \mathbb{R}$ and $\X \in \mathscr{P}$, we define the $t$-th root of $\X$ given by
\begin{equation*}
\X^{t} = \exp\left(t\log(\X)\right)\,.
\end{equation*}

\end{definition}

\begin{remark}

One can easily see that the cone $\mathscr{P}$ is not a group. However, 
\begin{itemize}
\item For all $\X \in \mathscr{P}$, $\X^{-1} \in \mathscr{P}$\,,
\item For all $\A\,, \B \in \mathscr{P}$, $\A^{-\frac{1}{2}}\B\A^{-\frac{1}{2}} \in \mathscr{P}$\,.
\end{itemize}
The second property will be useful in Kubo-Ando means theory\,.

\end{remark}

\noindent In the early 1980s, Kubo and Ando introduced a general axiomatic framework for binary operations on positive definite matrices and positive operators, now known as Kubo-Ando means \cite{KUBOANDO}. Their fundamental theorem establishes a one-to-one correspondence between operator means and operator monotone functions on the positive half-line, thereby providing a unified treatment of many classical matrix means, including the arithmetic, geometric, and harmonic means. Since then, Kubo-Ando theory has become one of the central tools in matrix analysis and operator theory, with numerous applications to operator inequalities, matrix geometry, quantum information theory, and mathematical physics. We now recall the definition of a Kubo-Ando mean in the matrix setting\,.

\begin{remark}

The cone $\mathscr{P}$ is naturally equipped with the Loewner partial order. More precisely, for all $\A\,, \B \in \mathscr{P}$, we write
\begin{equation*}
\A \preceq \B \qquad \text{ if } \qquad  \B - \A \geq 0\,,
\end{equation*}
that is, if the matrix $\B-\A$ is positive semi-definite. It is well known (see \cite{BHATIA}) that $\preceq$ defines a partial order on $\mathscr{P}$\,.

\label{RemarkLoewnerOrder}

\end{remark}

\begin{definition}

A Kubo-Ando mean on $\mathscr{P}$ is a continuous map
\begin{equation*}
\sigma: \mathscr{P} \times \mathscr{P} \mapsto \mathscr{P}
\end{equation*}
such that
\begin{enumerate}
\item If $\A \preceq \B$ and $\C \preceq \D$, then $\A \sigma \C \preceq \B \sigma \D$\,.
\item For all $g \in \GL_{n}(\mathbb{D})$ and $\A\,, \B \in \mathscr{P}$, we have
\begin{equation*}
g\left(\A \sigma \B\right)g^{*} = \left(g \A g^{*}\right) \sigma \left(g \B g^{*}\right)\,.
\end{equation*}
\item $\Id_{n} \sigma \Id_{n} = \Id_{n}$\,.
\end{enumerate}

\end{definition}

\noindent One of the main theorems of \cite{KUBOANDO} is summarized in the following theorem\,.

\begin{theo}

For all Kubo-Ando means $\sigma$ on $\mathscr{P}$, there exists a unique operator monotone function $f: \left(0\,, +\infty\right) \to \left(0\,, +\infty\right)$, with $f(1) = 1$, satisfying
\begin{equation}
\A \sigma \B = \A^{\frac{1}{2}} f\left(\A^{-\frac{1}{2}}\B\A^{-\frac{1}{2}}\right)\A^{\frac{1}{2}}\,, \qquad \left(\A\,, \B \in \mathscr{P}\right)\,.
\label{EquationKuboAndoSigma}
\end{equation}

\label{TheoremKuboAndoMean}

\end{theo}

\noindent In this paper, we will
\begin{itemize}
\item Extend Theorem \ref{TheoremKuboAndoMean} to the set of quaternionic hermitian matrices\,,
\item Define Kubo-Ando means for $\J$-Hermitian matrices and prove an analogue of Theorem \ref{TheoremKuboAndoMean} for these matrices\,.
\end{itemize}

\noindent We now recall some results of \cite{JOSEALLAN}. We denote by $\mathfrak{p}_{\J}$ the subset of $\Mat_{n}(\mathbb{D})$ given by
\begin{equation*}
\mathfrak{p}_{\J} = \left\{\X \in \Mat_{n}(\mathbb{D})\,, \X = \X^{\sharp}\right\}\,.
\end{equation*}
One of the main difference with the positive case (i.e. $\mathfrak{p}$) is that the restriction of $\exp$ to $\mathfrak{p}_{\J}$ is not injective. We will therefore replace $\exp$ by the $\J$-exponential map $\exp_{\J}$ introduced in \cite{JOSEALLAN}. We recall a key lemma (see \cite[Lemma~2.9]{JOSEALLAN})\,.

\begin{lemma}

The map 
\begin{equation*}
\Phi_{\J}: \mathfrak{p}_{\J} \to \mathfrak{p}\,, \qquad \X \mapsto \J\X\,,
\end{equation*}
is well-defined and bijective.

\label{LemmaBijectionPJP}

\end{lemma}

\noindent We denote by $\exp_{\J}: \mathfrak{p}_{\J} \to \Mat_{n}(\mathbb{D})$ the map given by
\begin{equation*}
\exp_{\J} := \Phi^{-1}_{\J} \circ \exp \circ \Phi_{\J}\,,
\end{equation*}
and let $\mathscr{P}_{\J} := \exp_{\J}(\mathfrak{p}_{\J})$. As explained in \cite[Section~2]{JOSEALLAN}, we have
\begin{equation*}
\mathscr{P}_{\J} = \left\{\X \in \mathfrak{p}_{\J}\,, \J\X > 0\right\}\,.
\end{equation*}
We denote by $\log_{\J}$ the inverse of $\exp_{\J}$\,.

\begin{remark}

Using Remark \ref{RemarkLoewnerOrder}, we define the $\J$-Loewner order on the cone $\mathscr{P}_{\J}$. More precisely, for all $\A\,, \B \in \mathscr{P}_{\J}$, we write
\begin{equation*}
\A \preceq_{\J} \B \qquad \Longleftrightarrow \qquad \J\A \preceq \J\B\,.
\end{equation*}
Since the map
\begin{equation*}
\mathscr{P}_{\J} \longrightarrow \mathscr{P}\,, \qquad \A \longmapsto \J\A\,,
\end{equation*}
is a bijection (see Lemma \ref{LemmaBijectionPJP}), it follows immediately that $\preceq_{\J}$ is a partial order on $\mathscr{P}_{\J}$\,.

\end{remark}

\begin{definition}

For all $t \in \mathbb{R}$ and $\X \in \mathscr{P}_{\J}$, we denote by $\X^{t}_{\J}$ the matrix in $\mathscr{P}_{\J}$ given by
\begin{equation*}
\X^{t}_{\J} = \exp_{\J}\left(t\log_{\J}(\X)\right)\,.
\end{equation*}
In particular, we have
\begin{equation*}
\X^{t}_{\J} = \J\left(\J\X\right)^{t}\,.
\end{equation*}

\end{definition}

\noindent What is surprising, compared with the positive case, is that in general, we have
\begin{equation}
\X^{t}_{\J}\X^{s}_{\J} \neq \X^{t+s}_{\J}\,, \qquad \left(s\,, t \in \mathbb{R}\,, \X \in \mathscr{P}_{\J}\right)\,.
\label{JSRoot}
\end{equation}
One way to solve this issue is to define a different product on $\mathscr{P}$. We denote by $\bullet$ the map
\begin{equation*}
\bullet: \Mat_{n}(\mathbb{D}) \times \Mat_{n}(\mathbb{D}) \to \Mat_{n}(\mathbb{D})
\end{equation*}
given by
\begin{equation*}
\A \bullet \B = \A\J\B\,, \qquad \left(\A\,, \B \in \Mat_{n}(\mathbb{D})\right)\,.
\end{equation*}
One can see that 
\begin{equation*}
\A \bullet \J = \A\J\J = \A\J^{2} = \A\,, \qquad \left(\A \in \Mat_{n}(\mathbb{D})\right)\,,
\end{equation*}
i.e. $\J$ is the identity element for $\bullet$. Moreover, one can see that a matrix $\A \in \Mat_{n}(\mathbb{D})$ is invertible for $\bullet$ if and only if $\A$ is invertible, and that
\begin{equation*}
\X^{s}_{\J} \bullet \X^{t}_{\J} = \X^{s+t}_{\J}\,, \qquad \left(s\,, t \in \mathbb{R}\,, \X \in \mathscr{P}_{\J}\right)\,.
\end{equation*}

\section{The cone of positive quaternionic hermitian matrices}

\label{SectionQuaternions}

We briefly recall some standard facts about the quaternion algebra. We refer the reader to \cite{ZHANG} for more details. The algebra of quaternions is defined by
\begin{equation*}
\mathbb{H} = \left\{a+b\mathbf{i}+c\mathbf{j}+d\mathbf{k}\,, a\,, b\,, c\,, d \in \mathbb{R}\right\}\,,
\end{equation*}
where
\begin{equation*}
\mathbf{i}^{2} = \mathbf{j}^{2} = \mathbf{k}^{2} = \mathbf{i}\mathbf{j}\mathbf{k} = -1\,.
\end{equation*}
In particular,
\begin{equation*}
\mathbf{i}\mathbf{j} = \mathbf{k}\,, \qquad \mathbf{j}\mathbf{k} = \mathbf{i}\,, \qquad \mathbf{k}\mathbf{i} = \mathbf{j}\,,
\end{equation*}
and
\begin{equation*}
\mathbf{j}\mathbf{i} = -\mathbf{k}\,, \qquad \mathbf{k}\mathbf{j} = -\mathbf{i}\,, \qquad \mathbf{i}\mathbf{k} = -\mathbf{j}\,.
\end{equation*}
Every quaternion can be written uniquely as
\begin{equation*}
q = z_{1}+z_{2}\mathbf{j}\,, \qquad \left(z_{1}\,, z_{2} \in \mathbb{C}\right)\,,
\end{equation*}
where $\mathbb{C}=\mathbb{R}\oplus\mathbb{R}\mathbf{i}$. Moreover,
\begin{equation*}
z\mathbf{j} = \mathbf{j}\overline{z}\,, \qquad \left(z \in \mathbb{C}\right)\,,
\end{equation*}
and the quaternionic conjugate is given by
\begin{equation}
\overline{q} = \overline{z_{1}} - z_{2}\mathbf{j}\,.
\label{ConjugationQuaternions}
\end{equation}
Finally,
\begin{equation*}
q\overline{q} = \overline{q}q = \left|z_{1}\right|^{2} + \left|z_{2}\right|^{2}\,,
\end{equation*}
which defines the Euclidean norm $\left|q\right| := \sqrt{q\overline{q}}$ of $q$\,.

\begin{remark}

The quaternion algebra admits the faithful embedding
\begin{equation*}
\Psi : \mathbb{H} \hookrightarrow \Mat_{2}(\mathbb{C})\,, \qquad z_{1}+z_{2}\mathbf{j} \longmapsto \begin{pmatrix} z_{1} & z_{2} \\ -\overline{z_{2}} & \overline{z_{1}} \end{pmatrix}\,.
\end{equation*}
More generally, if
\begin{equation*}
\A = \A_{1} + \A_{2}\mathbf{j} \in \Mat_{n}(\mathbb{H})\,,
\end{equation*}
where $\A_{1}\,, \A_{2}\in\Mat_{n}(\mathbb{C})$, then $\Psi$ extends to the embedding
\begin{equation*}
\Psi(\A) = \begin{pmatrix} \A_{1} & \A_{2} \\ -\overline{\A_{2}} & \overline{\A_{1}} \end{pmatrix} \in \Mat_{2n}(\mathbb{C})\,.
\end{equation*}

\end{remark}

\noindent We now briefly recall some standard facts about quaternionic matrices. We refer the reader to \cite{ZHANG} for more details. For all $n \in \mathbb{N}$, we denote by $\Mat_{n}(\mathbb{H})$ the algebra of $n \times n$ matrices with quaternionic entries and by $\GL_{n}(\mathbb{H})$ the group of invertible matrices. Since the quaternion algebra is not commutative, one has to distinguish between left and right eigenvalues. Throughout this paper, we only consider right eigenvalues. More precisely, a quaternion $\lambda \in \mathbb{H}$ is called a right eigenvalue of $\A \in \Mat_{n}(\mathbb{H})$ if there exists a non-zero vector $x \in \mathbb{H}^{n}$ such that
\begin{equation*}
\A x = x\lambda\,.
\end{equation*}
Unlike the real and complex cases, right eigenvalues are not uniquely determined. Indeed, if $\lambda$ is a right eigenvalue of $\A$ and $q \in \mathbb{H}^{\times}$, then $q^{-1}\lambda q$ is also a right eigenvalue of $\A$. Consequently, the spectrum of a quaternionic matrix is naturally regarded as a collection of conjugacy classes in $\mathbb{H}$.

\noindent For all $\A = \left(a_{ij}\right) \in \Mat_{n}(\mathbb{H})$, we define
\begin{equation*}
\A^{*} =\left(\overline{a_{j,i}}\right)\,,
\end{equation*}
where $\overline{q}$ denotes the quaternionic conjugate of $q \in \mathbb{H}$ (see \eqref{ConjugationQuaternions}). As in the real and complex case, we sat that a matrix $\A \in \Mat_{n}(\mathbb{H})$ is said to be Hermitian if
\begin{equation*}
\A=\A^{*}\,.
\end{equation*}
The spectral theory of quaternionic Hermitian matrices is remarkably similar to its real and complex counterparts. Indeed, every right eigenvalue of a Hermitian matrix is a real number. In particular, every conjugacy class reduces to a single point, so that the spectrum is well defined as a finite subset of $\mathbb{R}$.

\noindent We denote by $\U_{\mathbb{H}}$ the subgroup of $\GL_{n}(\mathbb{H})$ given by
\begin{equation*}
\U_{\mathbb{H}} := \left\{\U \in \GL_{n}(\mathbb{H})\,, \U\U^{*} = \Id_{n}\right\}
\end{equation*}
the quaternionic unitary group. Every Hermitian matrix $\A \in \Mat_{n}(\mathbb{H})$ admits a spectral decomposition
\begin{equation}
\A = \U \begin{pmatrix} \lambda_{1} & & \\ & \ddots & \\ & & \lambda_{n} \end{pmatrix}\U^{*}\,,
\label{SpectralDecompositionQuaternion}
\end{equation}
where $\U \in \U_{\mathbb{H}}$ and $\lambda_{1}\,, \ldots\,, \lambda_{n} \in \mathbb{R}$ are the eigenvalues of $\A$. Consequently, the functional calculus is defined exactly as in the complex case\,.

\noindent We denote by
\begin{equation*}
\mathfrak{p} := \mathfrak{p}_{\mathbb{H}} = \left\{\X \in \Mat_{n}(\mathbb{H})\,, \X=\X^{*}\right\}\,,
\end{equation*}
and by
\begin{equation*}
\mathscr{P} := \mathscr{P}_{\mathbb{H}} = \left\{\X \in \mathfrak{p}\,, \Spec(\X) \subseteq \left(0\,, +\infty\right)\right\}
\end{equation*}
the cone of positive quaternionic Hermitian matrices. As in the complex setting, the exponential map restricts to a bijection
\begin{equation*}
\exp : \mathfrak{p} \longrightarrow \mathscr{P}\,,
\end{equation*}
whose inverse is denoted by
\begin{equation*}
\log : \mathscr{P} \longrightarrow \mathfrak{p}\,.
\end{equation*}

\begin{remark}

We first observe that the Loewner order (see \cite{BHATIA}) extends naturally to the quaternionic setting. More precisely, for all $\A\,, \B \in \mathscr{P}$, we write
\begin{equation*}
\A \preceq \B \qquad \text{ if } \qquad  \B - \A \geq 0\,,
\end{equation*}
that is, if $\B-\A$ is positive semi-definite. Since $\mathscr{P}$ consists of Hermitian matrices with positive real spectrum, this defines a partial order exactly as in the real and complex settings\,.

\end{remark}

\noindent We are now ready to introduce the quaternionic analogue of Kubo-Ando means\,.

\begin{definition}

A Kubo-Ando mean on $\mathscr{P} = \mathscr{P}_{\mathbb{H}}$ is a continuous map
\begin{equation*}
\sigma: \mathscr{P} \times \mathscr{P} \longrightarrow \mathscr{P}
\end{equation*}
satisfying the following properties:
\begin{enumerate}
\item If $\A \preceq \B$ and $\C \preceq \D$, then $\A \sigma \C \preceq \B \sigma \D$\,.
\item For every $g\in\GL_{n}(\mathbb{H})$, and $\A\,, \B \in \mathscr{P}$, we get
\begin{equation*}
g\left(\A\sigma\B\right)g^{*} = \left(g\A g^{*}\right) \sigma \left(g\B g^{*}\right)\,.
\end{equation*}
\item $\Id_{n} \sigma \Id_{n} = \Id_{n}$\,.
\end{enumerate}

\label{DefinitionKuboAndoH}

\end{definition}

\noindent We first prove the following propositions\,.

\begin{proposition}

Let $\sigma$ be a Kubo-Ando mean on $\mathscr{P}$. Then there exists a unique function
\begin{equation*}
f: \left(0\,, +\infty\right) \longrightarrow \left(0\,, +\infty\right)
\end{equation*}
such that for every $a>0$, we have
\begin{equation}
\Id_{n} \sigma a\Id_{n} = f(a)\Id_{n}\,.
\label{RepresentingFunctionQuaternionic}
\end{equation}

\label{PropositionOne}

\end{proposition}

\begin{proof}

Fix $a>0$. By the congruence invariance of $\sigma$, for every $\U \in \U_{\mathbb{H}}$, we obtain
\begin{equation*}
\U\left(\Id_{n}\sigma a\Id_{n}\right)\U^{*} = \left(\U\U^{*}\right) \sigma \left(\U\left(a\Id_{n}\right)\U^{*}\right) = \Id_{n}\sigma a\Id_{n}\,.
\end{equation*}
Hence,
\begin{equation*}
\U\left(\Id_{n}\sigma a\Id_{n}\right) = \left(\Id_{n}\sigma a\Id_{n}\right)\U
\end{equation*}
for every $\U \in \U_{\mathbb{H}}$. Since the natural representation of $\U_{\mathbb{H}}$ on the right quaternionic vector space $\mathbb{H}^{n}$ is irreducible, the quaternionic version of Schur's lemma (see, for example, \cite{KNAPP}) implies that every quaternionic linear endomorphism commuting with the action of $\U_{\mathbb{H}}$ is a real scalar multiple of the identity. Therefore, there exists a unique real number $f(a)>0$ such that
\begin{equation*}
\Id_{n} \sigma a\Id_{n} = f(a)\Id_{n},
\end{equation*}
which proves the result\,.

\end{proof}

\begin{proposition}

Let $\sigma$ be a Kubo-Ando mean on $\mathscr{P}$, and let $f: \left(0\,, +\infty\right) \rightarrow \left(0\,, \infty\right)$ be the corresponding function obtained in Proposition \ref{PropositionOne}. Then, for every positive diagonal matrix
\begin{equation*}
\D = \diag\left(d_{1}\,, \ldots\,, d_{n}\right)\,,
\end{equation*}
we have
\begin{equation*}
\Id_{n} \sigma \D = \diag\left(f(d_{1})\,, \ldots\,, f(d_{n})\right)\,.
\end{equation*}

\label{PropositionTwo}

\end{proposition}

\begin{proof}

Let
\begin{equation*}
\P_{i} = \diag\left(0\,, \ldots\,, 0\,, 1\,, 0\,, \ldots\,, 0\right)
\end{equation*}
be the orthogonal projection onto the $i$-th coordinate. Since $\P_{i}\D=\D\P_{i}$, the congruence invariance of $\sigma$ gives
\begin{equation*}
\P_{i}\left(\Id_{n} \sigma \D\right)\P_{i} = \left(\P_{i} \sigma d_{i}\P_{i}\right) = f(d_{i})\P_{i}\,.
\end{equation*}
Therefore, for every $1 \leq i \leq n$, we get
\begin{equation*}
\left(\Id_{n} \sigma \D\right)_{i,i} = f(d_{i})\,.
\end{equation*}
Similarly, if $i\neq j$, applying the same argument to $\P_{i}+\P_{j}$ shows that the $\left(i\,, j\right)$-entry must vanish. Hence $\Id_{n} \sigma \D$ is diagonal, and
\begin{equation*}
\Id_{n} \sigma \D = \diag\left(f(d_{1})\,, \ldots\,, f(d_{n})\right)\,.
\end{equation*}

\end{proof}

\begin{remark}

In the proof of proposition \ref{PropositionTwo}, we use projection operators. Although $\P_{i} \notin \GL_{n}(\mathbb{H})$, congruence invariance extends to $\P_{i}$ by continuity: for $\varepsilon > 0$, the matrix $\G_{i}(\varepsilon) = \Id_{n} + \left(\varepsilon-1\right)\P_{i}$ is invertible with $\G_{i}(\varepsilon)\D\G_{i}(\varepsilon)^{*} \to \P_{i}\D\P_{i} + \left(\Id_{n} - \P_{i}\right)$ as $\varepsilon \to 0$, so applying the congruence-invariance axiom to $\G_{i}(\varepsilon)$ and letting $\varepsilon \to 0$, using the continuity of $\sigma$, yields the identity for $\P_i$ as well\,.

\end{remark}

\begin{corollary}

Let $\sigma$ be a Kubo-Ando mean with representing function $f$\,.
\begin{enumerate}
\item For every $\A \in \mathscr{P}$, we have
\begin{equation*}
\Id_{n} \sigma \A = f(\A)\,,
\end{equation*}
where $f(\A)$ is defined by the functional calculus\,.
\item For all $\A\,, \B\in\mathscr{P}$, we get
\begin{equation*}
\A \sigma \B = \A^{\frac{1}{2}}f\left(\A^{-\frac{1}{2}}\B\A^{-\frac{1}{2}}\right)\A^{\frac{1}{2}}\,.
\end{equation*}
\item The function $f$ is operator monotone with $f(1)=1$\,.
\end{enumerate}

\label{CorollaryOne}

\end{corollary}

\begin{proof}

\begin{enumerate}
\item Let $\A \in \mathscr{P}$. Using Equation \eqref{SpectralDecompositionQuaternion}, we get that $\A = \U\D\U^{*}$, with $\U \in \U_{\mathbb{H}}$ and $\D = \diag\left(\lambda_{1}\,, \ldots\,, \lambda_{n}\right)$. Using the congruence invariance of $\sigma$, it follows from Proposition \ref{PropositionTwo} that
\begin{eqnarray*}
\Id_{n} \sigma \A & = & \left(\U\U^{*}\right) \sigma \left(\U\D\U^{*}\right) = \U\left(\Id_{n} \sigma \D\right)\U^{*} \\ 
& = & \U f(\D)\U^{*} = f(\U\D\U^{*}) = f(\A)\,.
\end{eqnarray*}
\item Let $\A\,, \B\in\mathscr{P}$, and let $\X = \A^{-\frac{1}{2}}\B\A^{-\frac{1}{2}}$. Using that $\left(\A^{\frac{1}{2}}\right)^{*} = \A^{\frac{1}{2}}$, it follows from the congruence invariance that
\begin{eqnarray*}
\A \sigma \B & = & \left(\A^{\frac{1}{2}} \Id_{n} \A^{\frac{1}{2}}\right) \sigma \left(\A^{\frac{1}{2}}\X\A^{\frac{1}{2}}\right) = \A^{\frac{1}{2}}\left(\Id_{n} \sigma \X\right) \A^{\frac{1}{2}} \\ 
& = & \A^{\frac{1}{2}}f(\X)\A^{\frac{1}{2}} = \A^{\frac{1}{2}}f\left(\A^{-\frac{1}{2}}\B\A^{-\frac{1}{2}}\right)\A^{\frac{1}{2}}\,.
\end{eqnarray*}
\item If $\A \preceq \B$, then
\begin{equation*}
\Id_{n} \sigma \A \preceq \Id_{n} \sigma \B
\end{equation*}
by the monotonicity of $\sigma$. Hence
\begin{equation*}
f(\A) \preceq f(\B)\,,
\end{equation*}
which proves that $f$ is operator monotone.Moreover, 
\begin{equation*}
f(1)\Id_{n} = \Id_{n} \sigma \Id_{n} = \Id_{n}\,,
\end{equation*}
and therefore $f(1) = 1$\,.
\end{enumerate}

\end{proof}

\noindent We then obtain an analogue of Theorem \ref{TheoremKuboAndoMean} for the quaternionic cone $\mathscr{P}$\,.

\begin{theo}

The correspondence
\begin{equation*}
\sigma \longmapsto f\,,
\end{equation*}
where $f$ is defined by
\begin{equation*}
\Id_{n}\sigma a\Id_{n} = f(a)\Id_{n}\,, \qquad \left(a > 0\right)\,,
\end{equation*}
defines a one-to-one correspondence between Kubo-Ando means on $\mathscr{P}$ and operator monotone functions
\begin{equation*}
f: \left(0\,, +\infty\right) \longrightarrow \left(0\,, +\infty\right)
\end{equation*}
satisfying $f(1)=1$\,.

\label{TheoremKuboAndoMeanQuaternions}

\end{theo}

\begin{proof}

Corollary \ref{CorollaryOne} shows that every Kubo-Ando mean determines a unique normalized operator monotone function\,. Conversely, let
\begin{equation*}
f: \left(0\,, +\infty\right) \longrightarrow \left(0\,, +\infty\right)
\end{equation*}
be a normalized operator monotone function. Define
\begin{equation*}
\A \sigma_{f} \B = \A^{\frac{1}{2}}f\left(\A^{-\frac{1}{2}}\B\A^{-\frac{1}{2}}\right)\A^{\frac{1}{2}}\,.
\end{equation*}
Since the functional calculus for quaternionic Hermitian matrices satisfies the same properties as in the complex setting, one verifies exactly as in the classical Kubo-Ando theory that $\sigma_{f}$ satisfies the three axioms of Definition \ref{DefinitionKuboAndoH}. Hence $\sigma_{f}$ is a Kubo-Ando mean, whose representing function is precisely $f$. This proves the correspondence\,.

\end{proof}

\begin{remark}

All the definitions and results of Section~\ref{SectionPreliminaries} extend naturally to the quaternionic setting. More precisely, one defines
\begin{equation*}
\mathfrak{p}_{\J} = \left\{\X\in\Mat_{n}(\mathbb{H})\,, \X = \X^{\sharp}\right\}\,, \qquad \mathscr{P}_{\J} = \left\{\X \in \mathfrak{p}_{\J}\,, \J\X > 0\right\}\,.
\end{equation*}
Furthermore, the exponential maps restrict to bijections
\begin{equation*}
\exp: \mathfrak{p} \longrightarrow \mathscr{P}\,, \qquad \exp_{\J}: \mathfrak{p}_{\J} \longrightarrow \mathscr{P}_{\J}\,,
\end{equation*}
whose inverses are denoted by $\log$ and $\log_{\J}$, respectively. Consequently, every construction introduced in Section~\ref{SectionPreliminaries} extends verbatim to quaternionic matrices\,.

\end{remark}

\section{Kubo-Ando means on $\mathscr{P}_{\J}$}

\label{SectionKuboAndoPJ}

We now introduce the analogue of Kubo-Ando means on the cone $\mathscr{P}_{\J}$. Replacing the Loewner order by the $\J$-Loewner order and the adjoint by the $\J$-adjoint naturally leads to the following definition. Let $\mathbb{D} \in \left\{\mathbb{R}\,, \mathbb{C}\,, \mathbb{H}\right\}$\,.

\begin{definition}

A $\J$-Kubo-Ando mean on $\mathscr{P}_{\J}$ is a continuous map
\begin{equation*}
\sigma_{\J} : \mathscr{P}_{\J} \times \mathscr{P}_{\J} \longrightarrow \mathscr{P}_{\J}
\end{equation*}
such that
\begin{enumerate}
\item If $\A \preceq_{\J} \B$ and $\C \preceq_{\J} \D$, then $\A \sigma_{\J} \C \preceq_{\J} \B \sigma_{\J} \D$\,.
\item For every $g \in \GL_{n}(\mathbb{D})$ and every $\A\,, \B \in \mathscr{P}_{\J}$, we have
\begin{equation*}
g^{\sharp}\left(\A \sigma_{\J} \B\right)g = \left(g^{\sharp}\A g\right) \sigma_{\J} \left(g^{\sharp}\B g\right)\,.
\end{equation*}
\item $\J \sigma_{\J} \J = \J$\,.
\end{enumerate}

\end{definition}

\begin{lemma}

Let $\sigma_{\J}$ be a $\J$-Kubo-Ando mean on $\mathscr{P}_{\J}$. The binary operation $\sigma$ on $\mathscr{P}$ defined by
\begin{equation}
\X \sigma \Y = \Phi_{\J}\left(\Phi_{\J}^{-1}(\X) \sigma_{\J} \Phi_{\J}^{-1}(\Y)\right) = \J\left((\J\X)\sigma_{\J}(\J\Y)\right)
\label{MeanInducedByJMean}
\end{equation}
for all $\X\,,\Y\in\mathscr{P}$ is a Kubo-Ando mean on $\mathscr{P}$.

\label{LemmaOne}

\end{lemma}

\begin{proof}

Since $\Phi_{\J}$ and $\Phi_{\J}^{-1}$ are continuous, it follows immediately that $\sigma$ is continuous. Let $\X_{1}\,,\X_{2}\,,\Y_{1}\,,\Y_{2}\in\mathscr{P}$ be such that $\X_{1} \preceq\X_{2}$ and $\Y_{1}\preceq\Y_{2}$. By the definition of the $\J$-Loewner order, we have
\begin{equation*}
\J\X_{1} \preceq_{\J} \J\X_{2} \qquad \text{and} \qquad \J\Y_{1}\preceq_{\J}\J\Y_{2}\,.
\end{equation*}
The monotonicity of $\sigma_{\J}$ therefore gives
\begin{equation*}
\left(\J\X_{1}\right) \sigma_{\J} \left(\J\Y_{1}\right) \preceq_{\J} \left(\J\X_{2}\right) \sigma_{\J} \left(\J\Y_{2}\right)\,.
\end{equation*}
Applying $\Phi_{\J}$, we obtain
\begin{equation*}
\X_{1} \sigma \Y_{1} \preceq \X_{2}\sigma\Y_{2}\,.
\end{equation*}
Let now $g \in \GL_{n}(\mathbb{D})$. Since
\begin{equation*}
\J\left(g\X g^{*}\right) = \left(g^{*}\right)^{\sharp}\left(\J\X\right)g^{*}\,,
\end{equation*}
the congruence invariance of $\sigma_{\J}$ implies that
\begin{eqnarray*}
\left(g\X g^{*}\right) \sigma \left(g\Y g^{*}\right) & = & \J\left(\left(\J g\X g^{*}\right) \sigma_{\J} \left(\J g\Y g^{*}\right)\right) = \J\left(
\left(g^{*}\right)^{\sharp}\left(\left(\J\X\right) \sigma_{\J} \left(\J\Y\right)\right)g^{*}\right) \\
& = & g\J\left(\left(\J\X\right) \sigma_{\J} \left(\J\Y\right)\right)g^{*} =  g\left(\X \sigma \Y\right)g^{*}\,.
\end{eqnarray*}
Finally, since $\Phi_{\J}^{-1}(\Id_{n})=\J$, we have
\begin{equation*}
\Id_{n} \sigma \Id_{n} = \J\left(\J \sigma_{\J} \J\right) = \J^{2} = \Id_{n}\,.
\end{equation*}
Consequently, $\sigma$ is a Kubo-Ando mean on $\mathscr{P}$\,.

\end{proof}

\begin{lemma}

The map
\begin{equation*}
\sigma_{\J} \longmapsto \sigma
\end{equation*}
from the set of $\J$-Kubo-Ando means on $\mathscr{P}_{\J}$ to the set of Kubo-Ando means on $\mathscr{P}$, defined in Lemma \ref{LemmaOne} is injective\,.

\label{LemmaTwo}

\end{lemma}

\begin{proof}

Let $\sigma_{\J}^{(1)}$ and $\sigma_{\J}^{(2)}$ be two $\J$-Kubo-Ando means inducing the same Kubo-Ando mean $\sigma$ on $\mathscr{P}$. Let
$\A\,, \B\in\mathscr{P}_{\J}$. By the definition of the induced mean, we have
\begin{equation*}
\Phi_{\J}(\A) \sigma \Phi_{\J}(\B) = \Phi_{\J}\left( \A\sigma_{\J}^{(i)}\B\right)\,, \qquad \left(i = 1\,, 2\right)\,.
\end{equation*}
Hence,
\begin{equation*}
\Phi_{\J}\left(\A\sigma_{\J}^{(1)}\B\right) = \Phi_{\J}\left(\A\sigma_{\J}^{(2)}\B\right)\,.
\end{equation*}
Since $\Phi_{\J}$ is injective, it follows that
\begin{equation*}
\A\sigma_{\J}^{(1)}\B = \A\sigma_{\J}^{(2)}\B.
\end{equation*}
Therefore,
\begin{equation*}
\sigma_{\J}^{(1)} = \sigma_{\J}^{(2)}\,,
\end{equation*}
and the map $\sigma_{\J}\mapsto\sigma$ is injective\,.

\end{proof}

\begin{corollary}

There is a one-to-one correspondence between the set of $\J$-Kubo-Ando means on $\mathscr{P}_{\J}$ and the set of Kubo-Ando means on $\mathscr{P}$\,.

\end{corollary}

\begin{proof}

By the previous lemmas, every $\J$-Kubo-Ando mean on $\mathscr{P}_{\J}$ induces a Kubo-Ando mean on $\mathscr{P}$, and the resulting map is injective. Conversely, the same construction, applied to the inverse bijection
\begin{equation*}
\Phi_{\J}^{-1}: \mathscr{P} \longrightarrow \mathscr{P}_{\J}\,,
\end{equation*}
defines an injective map from the set of Kubo-Ando means on $\mathscr{P}$ to the set of $\J$-Kubo-Ando means on $\mathscr{P}_{\J}$. Moreover, these two constructions are inverse to one another. Hence, they define a bijective correspondence between the two sets of means\,.

\end{proof}

\begin{theo}

There is a one-to-one correspondence between the set of $\J$-Kubo-Ando means on $\mathscr{P}_{\J}$ and the set of normalized operator monotone functions
\begin{equation*}
f: \left(0\,, \infty\right) \longrightarrow \left(0\,, \infty\right)\,, \qquad f(1) = 1\,.
\end{equation*}
More precisely, if $\sigma_{\J}$ is a $\J$-Kubo-Ando mean and $f$ is the representing function of the corresponding Kubo-Ando mean on $\mathscr{P}$, then
\begin{equation}
\A \sigma_{\J }\B = \A^{\frac{1}{2}}_{\J} \bullet f_{\J}\left(\A^{-\frac{1}{2}}_{\J} \bullet \B \bullet \A^{-\frac{1}{2}}_{\J}\right) \bullet \A^{\frac{1}{2}}_{\J}
\label{RepresentationJMean}
\end{equation}
for every $\A\,, \B \in \mathscr{P}_{\J}$, where
\begin{equation*}
f_{\J}: \mathscr{P}_{\J} \longrightarrow \mathscr{P}_{\J}
\end{equation*}
is defined by
\begin{equation}
f_{\J}(\X) := \J f(\J\X)\,.
\label{DefinitionJFunctionalCalculus}
\end{equation}
Conversely, every normalized operator monotone function $f: \left(0\,, \infty\right) \rightarrow \left(0\,, \infty\right)$ defines a unique $\J$-Kubo-Ando mean by Equation \eqref{RepresentationJMean}\,.

\end{theo}

\begin{proof}

Let $\sigma_{\J}$ be a $\J$-Kubo--Ando mean on $\mathscr{P}_{\J}$. By the previous corollary, it induces a unique Kubo-Ando mean $\sigma$ on $\mathscr{P}$ given by
\begin{equation*}
\X \sigma \Y = \Phi_{\J}\left(\Phi_{\J}^{-1}(\X) \sigma_{\J} \Phi_{\J}^{-1}(\Y)\right)\,.
\end{equation*}
By the classical Kubo-Ando correspondence (see Theorems \ref{TheoremKuboAndoMean} and \ref{TheoremKuboAndoMeanQuaternions}), there exists a unique normalized operator monotone function
\begin{equation*}
f: \left(0\,, \infty\right) \longrightarrow \left(0\,, \infty\right)
\end{equation*}
such that
\begin{equation}
\X \sigma \Y = \X^{\frac{1}{2}}f\left(\X^{-\frac{1}{2}}\Y\X^{-\frac{1}{2}}\right)\X^{\frac{1}{2}}\,, \qquad \left(\X\,, \Y \in \mathscr{P}\right)\,.
\label{ClassicalRepresentationTransport}
\end{equation}
Let $\A\,, \B \in \mathscr{P}_{\J}$. Since
\begin{equation*}
\Phi_{\J}(\A) = \J\A\,, \qquad \text{and} \qquad \Phi_{\J}(\B) = \J\B\,,
\end{equation*}
we obtain
\begin{equation*}
\A \sigma_{\J} \B = \Phi_{\J}^{-1}\left(\Phi_{\J}(\A) \sigma \Phi_{\J}(\B)\right) = \J\left[\left(\J\A\right)^{\frac{1}{2}}f\left(\left(\J\A\right)^{-\frac{1}{2}}\left(\J\B\right)\left(\J\A\right)^{-\frac{1}{2}}\right)(\J\A)^{\frac{1}{2}}\right]\,.
\end{equation*}
Using that 
\begin{equation*}
\X \bullet \Y = \X\J\Y \qquad \text{ and } \qquad \A^{t}_{\J} = \J\left(\J\A\right)^{t}\,,
\end{equation*}
we get
\begin{equation*}
\J\left[\left(\J\A\right)^{-\frac{1}{2}}\left(\J\B\right)\left(\J\A\right)^{-\frac{1}{2}}\right] = \A^{-\frac{1}{2}}_{\J} \bullet \B \bullet \A^{-\frac{1}{2}}_{\J}\,.
\end{equation*}
Moreover, by the definition of $f_{\J}$,
\begin{equation*}
f_{\J}(\X) = \J f(\J\X)\,.
\end{equation*}
Using these identities in the preceding formulas gives
\begin{equation*}
\A \sigma_{\J} \B = \A^{\frac{1}{2}}_{\J} \bullet  f_{\J}\left(\A^{-\frac{1}{2}}_{\J} \bullet  \B \bullet  \A^{-\frac{1}{2}}_{\J}\right) \bullet \A^{\frac{1}{2}}_{\J}\,.
\end{equation*}
The uniqueness of $f$ follows from the uniqueness of the representing function of the induced mean $\sigma$ on $\mathscr{P}$\,.

\noindent Conversely, let $f: \left(0\,, \infty\right) \rightarrow \left(0\,, \infty\right)$ be a normalized operator monotone function. By the classical Kubo-Ando theorem, $f$ determines a unique Kubo-Ando mean $\sigma$ on $\mathscr{P}$. Transporting $\sigma$ through the bijection
\begin{equation*}
\Phi_{\J}: \mathscr{P}_{\J} \longrightarrow \mathscr{P}
\end{equation*}
defines a unique $\J$-Kubo-Ando mean $\sigma_{\J}$ on $\mathscr{P}_{\J}$. The computation above shows that this mean is given by Equation \eqref{RepresentationJMean}. Hence, the correspondence is bijective\,.

\end{proof}

\begin{remark}

The previous theorem allows us to transport every classical Kubo-Ando mean on $\mathscr{P}$ to a $\J$-Kubo-Ando mean on $\mathscr{P}_{\J}$. In particular, several familiar means admit simple closed expressions. For example, the arithmetic mean is preserved under the correspondence:
\begin{equation*}
\A \nabla_{\J} \B = \J\left(\left(\J\A\right) \nabla \left(\J\B\right)\right)  = \J\left(\frac{\J\A+\J\B}{2}\right) = \frac{\A + \B}{2}\,.
\end{equation*}

\noindent Similarly, the left-trivial mean, the parallel sum, and the harmonic mean satisfy
\begin{equation*}
\A \omega_{\ell, \J} \B = \A\,, \qquad \A :_{\J} \B = \left(\A^{-1}+\B^{-1}\right)^{-1}\,, \qquad \A !_{\J} \B = 2\left(\A^{-1}+\B^{-1}\right)^{-1}\,.
\end{equation*}
More generally, if $\M_{t}$ denotes the classical power mean associated with the representing function
\begin{equation*}
f_{t}(x) = \left(\frac{1+x^{t}}{2}\right)^{\frac{1}{t}}\,,
\end{equation*}
then its $\J$-analogue is given by
\begin{equation*}
\M_{\J,t}(\A\,, \B) = \A_{\J}^{\frac{1}{2}} \bullet \left(\frac{\J+\left(\A^{-\frac{1}{2}}_{\J} \bullet  \B \bullet \A^{-\frac{1}{2}}_{\J}\right)^{t}_{\J}}{2}
\right)^{\frac{1}{t}}_{\J} \bullet \A^{\frac{1}{2}}_{\J}\,.
\end{equation*}

\end{remark}

\begin{corollary}

The operator monotone function $f(x) = \sqrt{x}$ induces the $\J$-geometric mean
\begin{equation}
\A \sharp_{\J} \B = \A^{\frac{1}{2}}_{\J} \bullet \left(\A^{-\frac{1}{2}}_{\J} \bullet \B \bullet \A^{-\frac{1}{2}}_{\J}\right)^{\frac{1}{2}}_{\J} \bullet \A^{\frac{1}{2}}_{\J}\,, \qquad \left(\A\,, \B \in \mathscr{P}_{\J}\right)\,.
\label{JGeometricMean}
\end{equation}
Moreover, if $\A \bullet \B = \B \bullet \A$, then
\begin{equation}
\A \sharp_{\J} \B = \left(\A \bullet \B\right)^{\frac{1}{2}}_{\J}\,.
\label{JGeometricMeanCommuting}
\end{equation}

\end{corollary}

\begin{proof}

The first identity follows directly from the preceding theorem by taking $f(x)=\sqrt{x}$. Indeed, the induced $\J$-functional calculus is given by
\begin{equation*}
f_{\J}(\X) = \J\left(\J\X\right)^{\frac{1}{2}} = \X^{\frac{1}{2}}_{\J}\,.
\end{equation*}
Suppose now that $\A \bullet \B = \B \bullet \A$. Applying $\Phi_{\J}$ gives
\begin{equation*}
\left(\J\A\right)\left(\J\B\right) = \left(\J\B\right)\left(\J\A\right)\,.
\end{equation*}
Hence, the positive matrices $\J\A$ and $\J\B$ commute. Therefore, their classical geometric mean satisfies
\begin{equation*}
\left(\J\A\right) \sharp \left(\J\B\right) = \left(\left(\J\A\right)\left(\J\B\right)\right)^{\frac{1}{2}}\,.
\end{equation*}
By the definition of the transported mean, we obtain
\begin{eqnarray*}
\A \sharp_{\J} \B & = & \J\left(\left(\J\A\right) \sharp \left(\J\B\right)\right) = \J\left(\left(\J\A\right)\left(\J\B\right)\right)^{\frac{1}{2}} \\
& = & \J\left(\J\left(\A \bullet \B\right)\right)^{\frac{1}{2}} = \left(\A \bullet \B\right)^{\frac{1}{2}}_{\J}\,.
\end{eqnarray*}

\end{proof}

\section{The nearest J-Hermitian matrix}

We conclude this paper by characterizing the orthogonal projection onto the space of $\J$-Hermitian matrices with respect to the Frobenius norm\,.

\begin{remark}

We equip $\Mat_{n}(\mathbb{D})$ with the Frobenius norm
\begin{equation*}
\left\|\X\right\|_{\F} = \sqrt{\tr\left(\X^{*}\X\right)}\,, \qquad \left(\X \in \Mat_{n}(\mathbb{D})\right)\,.
\end{equation*}
Since $\J^{*}\J = \Id_{n}$, left multiplication by $\J$ preserves the Frobenius norm. More precisely,
\begin{equation*}
\left\|\J\X\right\|_{\F} = \left\|\X\right\|_{\F}\,, \qquad \left(\X \in \Mat_{n}(\mathbb{D})\right)\,.
\end{equation*}
Therefore, for every $\A\,, \B \in \Mat_{n}(\mathbb{D})$, we have
\begin{equation*}
\left\|\A-\B\right\|_{\F} = \left\|\J\A-\J\B\right\|_{\F}\,.
\end{equation*}

\end{remark}

\begin{proposition}

Let $\M \in \Mat_{n}(\mathbb{D})$. Then the unique matrix $\X \in \mathfrak{p}_{\J}$ minimizing the Frobenius distance
\begin{equation*}
\left\|\M - \X\right\|_{\F}
\end{equation*}
is given by
\begin{equation}
\X = \frac{\M+\M^{\sharp}}{2}\,.
\label{NearestJHermitian}
\end{equation}

\end{proposition}

\begin{proof}

By the previous remark,
\begin{equation*}
\left\|\M - \X\right\|_{\F} = \left\|\J\M - \J\X\right\|_{\F}\,.
\end{equation*}
Moreover,
\begin{equation*}
\X \in \mathfrak{p}_{\J} \qquad \Longleftrightarrow \qquad \J\X \in \mathfrak{p}\,.
\end{equation*}
Therefore, minimizing the distance from $\M$ to $\mathfrak{p}_{\J}$ is equivalent to minimizing the distance from $\J\M$ to the space $\mathfrak{p}$ of Hermitian matrices. It is well known (see, for example, \cite{BHATIA}) that the unique Hermitian matrix nearest to $\J\M$ is
\begin{equation*}
\frac{\J\M + \left(\J\M\right)^{*}}{2} = \frac{\J\M + \M^{*}\J}{2}\,.
\end{equation*}
Multiplying by $\J$ yields
\begin{equation*}
\X = \J\left(\frac{\J\M + \M^{*}\J}{2}\right) = \frac{\M + \J\M^{*}\J}{2} = \frac{\M+\M^{\sharp}}{2}\,,
\end{equation*}
which proves the result\,.

\end{proof}

\end{document}